\documentclass[12pt]{amsart}

\newcommand{\todo}[1]{\vspace{5 mm}\par \noindent
\marginpar{\textsc{ToDo}}
\framebox{\begin{minipage}[c]{0.95 \textwidth}
#1 \end{minipage}}\vspace{5 mm}\par}
\renewcommand{\todo}[1]{}  

\newcommand{\idiot}[1]{\vspace{5 mm}\par \noindent
\marginpar{\textsc{Note}}
\framebox{\begin{minipage}[c]{0.95 \textwidth}
#1 \end{minipage}}\vspace{5 mm}\par}
\renewcommand{\idiot}[1]{}  

\usepackage{amsthm}
\usepackage{amssymb}
\usepackage{amsfonts}
\usepackage{amsmath}

\numberwithin{equation}{section}
\theoremstyle{definition}
\newtheorem{definition}[equation]{Definition}
\theoremstyle{definition}

\theoremstyle{definition}
\newtheorem{example}[equation]{Example}
\theoremstyle{definition}
\newtheorem{notation}[equation]{Notation}
\theoremstyle{definition}

\theoremstyle{remark}

\theoremstyle{plain}
\newtheorem{theorem}[equation]{Theorem}
\theoremstyle{plain}
\newtheorem{lemma}[equation]{Lemma}
\theoremstyle{plain}
\newtheorem{proposition}[equation]{Proposition}
\theoremstyle{plain}

\theoremstyle{plain}
\newtheorem{corollary}[equation]{Corollary}
\theoremstyle{plain}

\theoremstyle{plain}

\DeclareMathOperator{\HF}{HF}

\DeclareMathOperator{\reg}{reg}

\DeclareMathOperator{\sqf}{sf}
\DeclareMathOperator{\HS}{HS}
\DeclareMathOperator{\gens}{gens}
\DeclareMathOperator{\spn}{span}
\newcommand{\field}{\Bbbk}
\newcommand{\Vxi}{V_{1}}
\newcommand{\Lxi}{L_{1}}
\newcommand{\Vhatxi}{V_{0}}
\newcommand{\Lhatxi}{L_{0}}
\newcommand{\Vone}{V_{0}}
\newcommand{\Lone}{L_{0}}
\newcommand{\Wone}{W_{0}}
\newcommand{\Wxi}{W_{1}}
\newcommand{\Wxj}{W_{1}}
\newcommand{\gdual}{Nnamztog}
\newcommand{\broom}{bottlebrush{}}
\newcommand{\Broom}{Bottlebrush{}}
\begin{document}

\author{Andrew H. Hoefel}
\address{Andrew H. Hoefel\\
Department of Mathematics and Statistics\\
Dalhousie University\\
Halifax, Nova Scotia B3H 3J5}
\email{andrew.hoefel@mathstat.dal.ca}
\author{Jeff Mermin}
\address{Jeff Mermin\\
Department of Mathematics\\
Oklahoma State University\\
Stillwater, OK 74078}
\email{mermin@math.okstate.edu}

\date{}
\title{Gotzmann squarefree ideals}
\maketitle

Abstract:  We classify the squarefree ideals which are Gotzmann in a
polynomial ring.

\section{Introduction}

Let $S=k[x_{1},\dots,x_{n}]$ and $R$ be either $S/(x_{1}^{2},\dots,
x_{n}^{2})$ or the exterior algebra on the same variables.

The question of what numerical functions can arise as the Hilbert functions of
homogeneous ideals these rings was answered by Macaulay \cite{M1927} (for
$S$) and Kruskal \cite{MR0154827} and Katona \cite{MR0290982} (for $R$).  
Macaulay's theorem and the Kruskal-Katona theorem have seen widespread
application 
in both commutative algebra and combinatorics.
These theorems provide lower bounds on 
the Hilbert function growth of homogeneous ideals, and describe
special ideals, called \emph{lexicographic} or \emph{lex} ideals,
which attain these bounds.  Lex ideals are defined combinatorially,
and are very well understood; for example, their minimal free
resolutions are known explicitly.  

Lex ideals are important tools in many contexts.  In geometry,
Hartshorne's proof that the Hilbert scheme is connected \cite{Ha} uses
lex ideals in an essential way.  More combinatorially, lex ideals and
Macaulay's and Kruskal and Katona's Theorems have arisen in Sperner
theory, network reliability, and other graph problems; see
\cite{MR1429390} or \cite{BL2005}.  

In many of these settings, the only relevant property of the lex
ideals is the slow growth of their Hilbert functions, so it is
worthwhile to consider other ideals whose Hilbert functions achieve
the bounds of Macaulay's (or Kruskal and Katona's) Theorem.  Such
ideals are called \emph{Gotzmann}.  Gotzmann ideals share many nice
properties with lex ideals; for example, they have componentwise
linear resolutions and maximal graded Betti numbers among ideals with
the same Hilbert function \cite{MR1684555}.  

However, only a few classes of Gotzmann ideals are known.  
Murai and Hibi show in \cite{MR2434473} that all Gotzmann ideals of
$S$ with at most $n$ generators have a very specific form.
The problem of understanding monomial Gotzmann ideals has proven
slightly more tractable.  Bonanzinga classifies the principal Borel
ideals generated in degree at most four which are
Gotzmann \cite{Bo2003}.
Mermin enumerates in \cite{MR2253664} the lexlike ideals, ideals which
are generated by initial segments of ``lexlike'' sequences and which
share many properties with lex ideals, including minimal Hilbert
function growth. 
 Murai studies Hilbert functions for which the only Gotzmann monomial ideals
 are lex, and 
classifies the Gotzmann monomial ideals of $k[a,b,c]$, in \cite{MR2379725}, and
Olteanu, Olteanu, and Sorrenti
\cite{OOS2008} classify the Gotzmann ideals which are generated by
(not necessarily initial) segments in the lex order.  
Finally, in \cite{H2009}, Hoefel shows that a
graph has a Gotzmann edge ideal if and only if it is star-shaped.
In this paper,
we generalize Hoefel's result by classifying all the Gotzmann ideals
of $S$ which are generated by squarefree monomials.  An immediate
consequence of our classification is that all such ideals have at most
$n$ generators, so they have the form prescribed by Murai and Hibi.

In section 2, we define notation which will be used throughout the
paper, and recall background information about Gotzmann ideals,
squarefree ideals, and the relationship between $R$ and $S$.
In section 3, we classify the squarefree Gotzmann ideals of $S$.
Finally, in section 4, we begin to study the monomial Gotzmann ideals of $R$.
Our main result is that a Gotzmann ideal has Gotzmann Alexander dual
if and only if all of its degreewise components are lex in some order.

\textbf{Acknowledgements:}  We thank Chris Francisco, Huy T\`ai H\`a and Gwyn Whieldon 
for helpful discussions. The first author also thanks NSERC and the Killam Trusts for their
financial support.

\section{Background and Notation}

Let $S=\field[x_{1},\dots,x_{n}]$ be the polynomial ring and
$R=S/(x_{1}^{2},\dots, x_{n}^{2})$ be the associated ``squarefree ring''.
Let $\mathbf{m}=(x_{1},\dots,x_{n})$ be the homogeneous maximal ideal.

In this paper all ideals are homogeneous. A \emph{monomial ideal} is 
an ideal with a generating set consisting of monomials.
Every monomial $I$ has a canonical minimal set of monomial generators
which we denote $\gens(I)$. When we refer to the generators of a 
monomial ideal --- for example, to count them --- we mean those in 
this canonical generating set.  If all the generators of $I$ are 
squarefree monomials, then we say that $I$ is \emph{squarefree}.  

For an ideal $I$, we write $I_{d}$ for the vector space of degree $d$
forms in $I$. When $I$ is a monomial ideal, a basis for $I_{d}$ is given
by the monomials of degree $d$ that are in $I$.

An important vector space is $S_{d}$ (or $R_{d}$), the space of all
forms of degree $d$.
We call subspaces of $S_d$ or $R_d$ \emph{monomial vector spaces} when they have
a monomial basis. The (unique) monomial basis of a monomial vector space $V_{d}$
will be denoted $\gens(V_{d})$.  Usually, we indicate the degree of a
monomial vector space with a subscript in this way.

The ideal $(V_{d})$ generated by a monomial vector space $V_{d} \subseteq
S_d$
is a monomial ideal 
with with $\gens((V_{d})) = \gens (V_{d})$. We will often need to consider the
monomial vector space  
\[ \mathbf{m}_1 V_{d} = \spn_{\field}  \{ x_i m \mid m \in \gens(V_{d})\}.
\]
While we treat this as a product of monomial vector spaces, it has a
natural interpretation in terms of ideals.  If $I=(V_{d})$ is the
ideal generated by $V_{d}$, then
$\mathbf{m}_1 V_{d} = I_{d+1}$.

We will write $|V_{d}|$ or sometimes $|V|_{d}$ to denote the vector
space dimension of $V_{d}$. 

\begin{definition}
The \emph{Hilbert function} of an ideal $I$ is the function $\HF_{I}: \mathbb N \to \mathbb N$
which gives the dimension of each component of $I$:
\[	\HF_{I}(d) = |I_d|.
\]
\end{definition}

When $I$ is a monomial ideal $|I_d| = |\gens(I_d)|$ is simply the
number of degree $d$ monomials in $I$.  If $I$ is squarefree, we write
$I^{\sqf}=IR$ for the corresponding ideal 
of $R$. Thus $|I^{\sqf}_d|$ is the number of squarefree monomials in
$I$ of degree $d$. 

The \emph{Hilbert series} and \emph{squarefree Hilbert series} of $I$
are $\HS_{I}(t)=\sum_{d=0}^{\infty}|I_{d}|t^{d}$ and
$\HS_{I}^{\sqf}(t)=\sum_{d=0}^{\infty} |I_{d}^{\sqf}|t^{d}$.  If $I$ is
squarefree, these are related by the formula
$\HS_{I}(t)=\HS_{I}^{\sqf}(\frac{t}{1-t})$.

\begin{definition}
We say that an ideal $I$ of $S$ (or, with the obvious changes, of $R$)
is \emph{Gotzmann} if for all degrees $d$ and all ideals $J$ satisfying $|I_d| = |J_d|$
we have $|\mathbf{m}_1 I_d| \leq |\mathbf{m}_1 J_d|$.
\end{definition}

The Gotzmann property may be viewed degreewise as a property of vector spaces.
We say that a vector space $V_{d} \subseteq S_d$ (or $R_d$) is \emph{Gotzmann}  
if for all subspaces $W_{d} \subseteq S_d$ (resp. $R_d$) with $|V_{d}| = |W_{d}|$
we have $|\mathbf{m}_1 V_{d}| \leq |\mathbf{m}_1 W_{d}|$. Thus an ideal $I$ is
Gotzmann if and only if each component $I_{d}$ is Gotzmann.

\begin{definition}
A vector space $L_{d} \subseteq S_d$ (or $R_d$) is a \emph{lex segment} if 
it is spanned by an initial segment of the degree $d$ monomials in
lexicographic order. 
An ideal is \emph{lex} if each of its components are lex segments.
We say that a squarefree ideal $L$ is \emph{squarefree lex} if
$L^{\sqf}$ is lex in $R$. 
\end{definition}

Lex ideals are important because they are very well understood
combinatorially, and because of the following theorem of Macaulay
\cite{M1927} (in $S$) and Kruskal and Katona \cite{MR0154827, MR0290982} (in $R$):

\begin{theorem}[Macaulay, Kruskal, Katona]
Lex ideals are Gotzmann.
\end{theorem}

\begin{corollary}
For every homogeneous ideal $I$ there is a unique lex ideal $L$ with
$\HF(L)=\HF(I)$. 
\end{corollary}

Macaulay's theorem allows the following alternative characterization
of Gotzmann ideals: 
\begin{proposition} \label{gotzByLexSeg}
A degree $d$ monomial vector space $V_{d}$ is Gotzmann if and only if $|\mathbf{m}_1 V_{d}| = |\mathbf{m}_1 L_{d}|$ where $L_{d}$ is the degree $d$ lex segment of dimension $|V_{d}|$.
\end{proposition}

\begin{corollary}\label{gotzByGenerators}
Let $L$ be the lex ideal with the same Hilbert function as $I$.
Then $I$ is Gotzmann if and only if $L$ and $I$ have the same number
of generators in each degree. 
\begin{proof}
The number of degree $d$ generators of $I$ and $L$ is given by 
$|I_{d}|-|\mathbf{m}_1 I_{d-1}|$  
and $|L_d| - |\mathbf{m}_1 L_{d-1}|$, respectively.
Since $I$ and $L$ have the same Hilbert function, they
have the same number of generators in degree $d$ if and only if
$I_{d-1}$ is Gotzmann. 
\end{proof}
\end{corollary}

Gotzmann's persistance theorem \cite{MR0480478} states that if $V_{d}$ is a
Gotzmann vector space then 
$\mathbf{m}_1 V_{d}$ is also Gotzmann. Thus, to check if an ideal is
Gotzmann, one only needs  
to check that its components are Gotzmann in the degrees of its
minimal generators. 

\begin{theorem}[Gotzmann's persistence theorem, vector space version]
Suppose that $I_{d}$ is a Gotzmann vector space.  Then $m_{1}I_{d}$ is
Gotzmann.
\end{theorem}

\begin{theorem}[Gotzmann's persistence theorem, ideal version]
Suppose that every generator of $I$ has degree at most $d$, and let $L$ be the
lex ideal with the same Hilbert function as $I$.  If $L$ has no
generators of degree $d+1$, then all generators of $L$ have degree at
most $d$.  In particular, if $I$ and $L$ have the same number of
generators in every degree less than or equal to $d+1$, then $I$ is Gotzmann.
\end{theorem}

The persistence theorem holds in both $S$ and $R$ \cite{MR1444495, MR879339}.  
\todo{
	Jeff: AHH (gotzmann theorems) cites Furedi-Griggs for the persistence
	theorem for simplicial complexes, then gives a (different?) proof
	over E.  Hunt that down. 
	Andrew: Did it. They're right. It's theorem 2.1 of FG. Look at example 2.4 of FG.
	They have a definition of a jumping number. We're concerned with the case (g=k-1).
	Their example shows that if you take a Gotzmann monomial vectorspace of dimension m
	with $m = \binom{a_k}{k} + \cdots + \binom{a_1}{1}$ then you can drop $a_1$ generators
	and still be Gotzmann. Perhaps we should list this as a Gotzmann construction; it
	will work componentwise as well.
        Jeff:  I totally believe that, and it's worth
        understanding, but we can't do anything with it at this time;
        it's out of place in this paper.
	}

\section{Gotzmann Squarefree Ideals of the Polynomial Ring}

We will classify the squarefree ideals of $S$ which are Gotzmann.  
To do this, we compare squarefree ideals with their squarefree 
lexifications and exploit the interaction
between $S$ and $R$.

In \cite{H2009}, Hoefel proved that a squarefree quadric ideal is
Gotzmann if and only if it is the edge ideal of a \emph{star-shaped}
graph.  We generalize this result as follows:  

\begin{definition}
Let $H$ be a pure $d$-dimensional hypergraph.  We say that $H$ is
\emph{star-shaped} if there exists a $(d-1)$-simplex which is contained in
every edge of $H$.  More generally, we say that a $d$-dimensional
simplicial complex $\Delta$ is a \emph{supernova} if there exists a 
chain of faces $\varnothing\subset F_{0}\subset F_1 \subset \cdots \subset F_{d-1}$ such that every $i$-dimensional 
facet of $\Delta$ contains the $(i-1)$-dimensional face $F_{i-1}$.
\end{definition}

We show in Theorem \ref{classify} that a squarefree ideal is Gotzmann
if and only if it is the edge ideal of a supernova.  In particular, a
purely $d$-generated squarefree ideal is 
Gotzmann if and only if it is the edge ideal of a $(d-1)$-dimensional
supernova.

\todo{Decide if ``supernova'' is a good word or not.  Other
  candidates:  ``flower'', ``generalized star-shaped''.}

A consequence of Theorem \ref{classify} is that all Gotzmann
squarefree ideals of $S=\field[x_{1},\dots, x_{n}]$ have at most $n$
generators.  The Gotzmann ideals of $S$ with at most $n$ generators
are classified by Murai and Hibi \cite{MR2434473}; it is clear from their
classification that any Gotzmann squarefree ideal with at most $n$
generators must have the form prescribed by Theorem \ref{classify}.
Thus, if this bound on the number of generators could be easily
proved, Theorem \ref{classify} would be a corollary of 
\cite[Theorem 1.1]{MR2434473}. We have been unable to find a simple proof of this bound.
Regardless, the smaller scope of our investigation allows a simpler
proof than that given in \cite{MR2434473}.

\begin{definition}
The \emph{squarefree lexification} of a squarefree ideal $I \subseteq S$ is the
squarefree lex ideal $L$ in $S$ with the same Hilbert function as $I$.
\end{definition}

The existence of squarefree lexifications follows from the following
construction: 
Let $J^{\sqf} \subseteq R$ be the lex ideal having the same Hilbert function
as $I^{\sqf}$.  
Then let $L$ be the ideal of $S$ with $L^{\sqf} =
J^{\sqf}$ (that is, $L$ is generated by the monomials of
$J^{\sqf}$). Then $L$ is squarefree lex and has
the same Hilbert function as $I$ because
\[ \HS_{I}(t)=\HS_{I^{\sqf}}(\frac{t}{1-t}) = \HS_{J^{\sqf}}(\frac{t}{1-t}) = \HS_{L}(t).
\]

The following proposition is a consequence of Kruskal and Katona's theorem:

\begin{proposition}[Aramova, Avramov, Herzog \cite{MR1603874}]\label{gotzThenSfGotz}
If $I \subseteq S$ is a Gotzmann squarefree ideal then $I^{\sqf}$ is
Gotzmann in $R$. 
\end{proposition}
\todo{Hunt down this reference.  I couldn't find it quickly.  If it
  doesn't exist, the proof is as follows:  Let $L^{\sqf}$ be the
  squarefree lexification, and look at growth from degree $d$ to
  $d+1$.  Then $L^{\sqf}$ grows more slowly in $R$, since it's lex,
  while $I$ grows more slowly in $S$, since it's Gotzmann.  Thus, they
  grow at the same speed by the $h$-vector thing.

  It's proposition 2.3 of ``resolutions of monomial ideals and cohomology over exterior algebras.''
}

\begin{lemma} \label{sqflexIsGotz}
If $I \subseteq S$ is a Gotzmann squarefree ideal then its squarefree
lexification $L$ is Gotzmann. 
\begin{proof}
By Proposition \ref{gotzThenSfGotz}, $I^{\sqf}$ is Gotzmann in $R$.
Thus, applying Corollary \ref{gotzByGenerators}, $I^{\sqf}$ and
$L^{\sqf}$ have the same number of minimal generators in 
every degree.  Now $I$ and $I^{\sqf}$ have the same generating set, as
do $L$ and $L^{\sqf}$, so $I$ and $L$ have the same number of
generators in every degree.  Applying Corollary \ref{gotzByGenerators}
again, $L$ must be Gotzmann in $S$.  
\idiot{This is the old proof:
\begin{proof}
It will be shown that $|\mathbf{m}_1 I_d| =|\mathbf{m}_1 L_d|$ for an arbitrary degree $d$.

Let $J$ and $K$ be the square free monomial ideals generated 
by the generators of $I$ and $L$ respectively that have 
degrees less than or equal to $d$.
Then $J_{d+1} = \mathbf{m}_1 I_d$ and $K_{d+1} = \mathbf{m}_1 L_d$. 

It will now be shown that $JR$ and $KR$ have the same Hilbert function as ideals in $R$. In degrees $k \leq d$, $|(JR)_k| = |(IR)_k| = |(LR)_k| = |(KR)_k|$ by
the construction of $L$. As $J_d$ is Gotzmann, $(JR)_d$ is Gotzmann using the previous proposition.
Also $(KR)_d$ is Gotzmann as it is a lex segment. Since $JR$ is generated in
degrees $d$ and less, Gotzmann's persistence theorem shows that $|(JR)_k| = |(KR)_k|$ for $k > d$.

Finally, $|J_{d+1}| = |\mathbf{m}_1 I_d|$ and $|K_{d+1}| = |\mathbf{m}_1 L_d|$ can be determined by the Hilbert series of $JR$ and $KR$ respectively. However, since these these two Hilbert series are equal, we have $|\mathbf{m}_1 L_d| = |\mathbf{m}_1 I_d|$.
\end{proof}
}
\end{proof}
\end{lemma}

\idiot{
\todo{
Make Jeff double check that the following didn't require $I$ Gotzmann.
}It doesn't.}

The next general lemma is the squarefree version of a common inductive
technique for studying Hilbert functions in terms of lex ideals.  
\todo{Determine if this prose is appropriate.  There are similar
  looking lemmas in [Me, lexlike sequences] and, if I read it right on
  skimming, [MH].
}

\begin{lemma}\label{inductOnNumVars}
 Let $I \subseteq S$ be a squarefree ideal and let $L$ be its squarefree
lexification. Then $L \subseteq (x_{1})$ if and only if $I \subseteq (x_{i})$ for some variable $x_{i}$.
\end{lemma}
\begin{proof}
If $I \subseteq (x_i)$ then $I^{\sqf} \subseteq (x_i)^{\sqf}$ and hence
\[ |L_d^{\sqf}| = |I_d^{\sqf}| \leq |(x_i)_d ^{\sqf}| = |(x_1)_d^{\sqf}|.
\]
As $(x_1)^{\sqf}_d$ is a lex segment in $R$, we have $L_d^{\sqf} \subseteq (x_1)_d^{\sqf}$ 
and hence every generator of $L$ is divisible by $x_1$.

Conversely, assume that $L \subseteq (x_1)$. We have
$|L_{n-1}^{\sqf}|\leq n-1$, so there is at least one squarefree
monomial $m$ of degree $n-1$ which is not in $I$.  Write 
$m=\frac{x_{1}\cdots x_{n}}{x_{i}}$.  We claim that $I\subseteq
(x_{i})$.  Indeed, every squarefree monomial outside of $(x_{i})$ divides $m$, so
no such monomial can appear in $I^{\sqf}$.  Thus every generator of $I$
is in $(x_{i})$ and, in particular, $I\subseteq (x_{i})$.
\end{proof}

\begin{lemma}\label{linform}
If $I \subseteq S$ is a squarefree Gotzmann ideal then 
either $I \subseteq(x_{i})$ for some variable $x_i$ or $(x_{i}) \subseteq I$ 
for some variable $x_i$.
\begin{proof}
Suppose to the contrary that $I$ is Gotzmann but, for all $i$, $I \not
\subseteq (x_i)$ and $(x_i) \not \subseteq I$.  We will show that $L$,
the squarefree lexification of $I$,
is not Gotzmann, contradicting Lemma \ref{sqflexIsGotz}. 

It follows from Lemma \ref{inductOnNumVars} that $L\not\subseteq
(x_{1})$. Therefore we 
may choose a generator $m$ of $L$ which is not divisible by $x_1$. Let
$d$ be the degree of $m$. 

Since $I$ contains no variable, $L$ cannot contain $x_1$.  Choose a 
 squarefree monomial $m'\in (x_{1})\smallsetminus L$ of maximal degree
 $d'$.
As $L$ is squarefree lex and $m$ is not divisible by $x_1$,
$L$ contains all squarefree monomials that are divisible by $x_1$ and
have degree $d$ or larger. Thus, $d' < d$.

Let $T$ be the ideal generated by $\gens(L)\cup \{x_{1}^{d-d'}m'\}\setminus \{m\}$. Note that $|T_d| = |L_d|$.
\idiot{
The generators of $L$ are square-free and the support of $x_1^{d-d'}m'$ is
exactly the suppport of $m'$ which is not in $L$.
}

Let $A = \gens(\mathbf{m}_1 L_d)$ and $B = \gens(\mathbf{m}_1 T_d)$ be
the sets of degree $d+1$ monomials lying above $L_d$ and $T_d$
respectively.  If $L$ were Gotzmann, it would follow that $|A|\leq
|B|$.  We will show that instead $|B|<|A|$.

We claim that $B\smallsetminus A= \{x_{1}^{d-d'}m'x_{i} : x_{i} \mbox{
  divides }m'\}$.
Indeed, let $\mu\in B\smallsetminus A$ be a monomial.  Then $\mu$ is
divisible by 
$x_{1}^{d-d'}m'$, so it has the form $x_1^{d-d'} m' x_i$ for some $i$.
If $x_i$ divides $m'$ then the support of $x_1^{d-d'} m' x_i$ is $m'$
and hence $\mu$ is not in $A$. 
On the other hand,
if $x_i$ does not divide $m'$, then $m' x_i$ is a squarefree monomial
of degree $d' +1$ which is divisible by $x_1$. By the choice of $m'$,
we have $m' x_i \in L$ and hence $x_1^{d-d'} m' x_i \in A$, proving
the claim.
In particular, $|B\smallsetminus A|=d'$.

Similarly, monomials in $A \smallsetminus B$ must have the form $x_i
m$ for some $i$. If $x_i$ divides $m$ then $x_i m$ has support $m$ and
hence is not in $B$. Thus  
\[ A \smallsetminus B \supseteq \{x_{i}m:x_{i} \mbox{ divides } m\}
\] 
which has cardinality at least $d$. 

As $|B \smallsetminus A| = d' < d \leq |A \smallsetminus B|$, it follows that
$|\mathbf{m}_1 T_d| = |B| < |A| = |\mathbf{m}_1 L_d|$, and so $L$ is
not Gotzmann. 
\end{proof}
\end{lemma}

\begin{lemma} Let $I \subseteq S$ be a Gotzmann squarefree monomial
  ideal with $I \subseteq (x_i)$. 
Then $\frac{1}{x_i} I$ is Gotzmann in $S$.
\begin{proof}
Let $L$ be the (non-squarefree) lexification of $I$.
It is clear that $L \subseteq (x_1)$:  $(x_{1})$ is the lexification
of $(x_{i})$, which contains $I$.  

Now multiplication by $x_{i}$ is a degree one module isomorphism from
$\frac{1}{x_{i}}I$ to $I$, and similarly for $L$.  Applying Corollary
\ref{gotzByGenerators} twice, we have that $\frac{1}{x_{i}}I$ is
Gotzmann with the same Hilbert function as $\frac{1}{x_{1}}L$. 
\end{proof}
\end{lemma}

\begin{lemma}
Let $I \subseteq S$ be a Gotzmann squarefree monomial ideal with $(x_i) \subseteq I$.
The image of $I$ in the quotient ring $S/(x_i)$ is a Gotzmann squarefree monomial ideal.
\begin{proof}
By renaming the variables if necessary, we may assume that $(x_1) \subseteq I$.
Let $\bar I$ be the image of $I$ in $S/(x_1)$ (or, equivalently,
the squarefree monomial ideal of $\field[x_2, \ldots, x_n]$ generated by
every generator of $I$ other than $x_1$).

Let $L$ be the (non-squarefree) lexification of $I$ in $S$.
We have $(x_1) \subseteq L$. Let $\bar L$ be the image of $L$ in $S/(x_1)$. 
Then $\bar L$ is the lexification of $\bar I$. 
Observe that $\gens(\bar{I})=\gens(I)\smallsetminus\{x_{1}\}$ and
similarly for $L$.  Thus, applying Corollary \ref{gotzByGenerators}
twice, $\bar{I}$ is Gotzmann.
\end{proof}
\end{lemma}

Lemma \ref{linform} allows us to characterize the squarefree ideals
which are Gotzmann.

\begin{theorem}\label{classify}
Suppose $I \subseteq S$ is a squarefree ideal.  Then $I$ is Gotzmann
if and only if 
\[ 
I = 
	m_{1}(x_{i_{1,1}},\ldots,x_{i_{1,r_{1}}}) + 
	m_{1}m_{2}(x_{i_{2,1}},\ldots,x_{i_{2,r_{2}}}) 
	+ \cdots +  
	m_{1} \cdots m_{s}
	(x_{i_{s,1}},\ldots,x_{i_{s,r_{s}}})
\]
for some squarefree monomials $m_{1},\dots, m_{s}$ and 
variables $x_{i,j}$ all having pairwise disjoint support.
\begin{proof}
\idiot{
Suppose that $I$ has the given form.  Then $I$ is a lexlike ideal (see
\cite{MR2253664}).
Suppose that $I$ has the given form.  Then $I$ is the polarization of
a universally lex ideal.  (See \cite{Fa} for polarization and
\cite{MH} for universally lex ideals.)
Suppose that $I$ has the given form.  Then $I$ is a Gotzmann ideal of
type mumble (see \cite{MH} or \cite{Go}).

Now
}
Suppose that $I$ is Gotzmann.  By Lemma \ref{linform}, either $(x_j) \subseteq I$ or $I \subseteq (x_j)$ for some $j$. 

If $I \subseteq (x_j)$ then $\frac{1}{x_{j}}I$ is Gotzmann in $S$ and its
generators are supported on $\{x_1, \ldots, \hat x_j, \ldots, x_n\}$.
Inducting on the number of variables, $\frac{1}{x_j} I$ may be written as
\[
	m_{1}(x_{i_{1,1}},\ldots,x_{i_{1,r_{1}}}) + 
	m_{1}m_{2}(x_{i_{2,1}},\ldots,x_{i_{2,r_{2}}}) 
	+ \cdots +  
	m_{1} \cdots m_{s}
	(x_{i_{s,1}},\ldots,x_{i_{s,r_{s}}})
\]
where $x_j$ does not appear in this expression.
Thus, $I$ can be expressed in the desired form
by replacing $m_1$ with $x_j m_1$.

Alternatively, suppose that $(x_j) \subseteq I$, so,
without loss of generality, $I=(x_j)+J$, where $J$ is Gotzmann in the 
ring $\field[x_1,\dots, \hat x_j, \ldots, x_{n}]$. By induction on
the number of variables, $J$ may be written in the desired form
and so $I = (x_j) + J$ has the desired form as well (with $m_1 = 1$).
\end{proof}
\end{theorem}

\idiot{Hoefel shows in \cite{H2009} that a quadric squarefree ideal is Gotzmann
if and only if it is the edge ideal of a star-shaped graph.  We will
describe Gotzmann squarefree ideals in terms of the more general
\broom{} graphs.  \Broom{} graphs are studied in
\cite{Bouchat:2010fk}, where they are called broom graphs.  
We have changed the name so that ``broom graph'' may be reserved for
graphs where the bristles connect only to the bottom of the handle.

\todo{Find graph theory literature that already uses ``broom graph''.}

\begin{definition}
A graph $G=(V,E)$ is called a \emph{\broom{} graph} if the vertices can
be partitioned $V=Y\coprod X$ with $Y=\{y_{0},\dots, y_{s}\}$, and
every edge has one of the two forms $(y_{i},y_{i+1})$ or
$(x_{i},y_{j})$ with $x_{i}\in X$.  The subgraph on $Y$ is called the
\emph{handle}, and the edges $(x_{i},y_{j})$ are called the
\emph{bristles}.  In order that the handle be well-defined, we require
that there be at least one bristle of the form $(x_{i},y_{s})$.  We
say that $s$ is the \emph{height} of $G$.  
\end{definition}

\todo{an example would be good here.  Maybe I'll xfig one later.}

\begin{definition}
A \emph{labeled rooted \broom{} graph} is a \broom{} graph as above where
the vertex $y_{0}$ is called the \emph{root} and labeled by the
monomial $1$, and every other vertex is labeled by a distinct variable
of $S$.
\end{definition}

Given a broom graph, we associate a monomial to every path from the root to a
bristle by multiplying the labels of its vertices.  By Theorem
\ref{classify}, these monomials generate a Gotzmann ideal.  We
formulate this more precisely as follows.

\begin{definition}
Let $G=(V,E)$ be a rooted tree. The \emph{rooted path complex} of $G$ is the
simplicial complex 
\[
\mathcal{P}(G)=\{S\subseteq V \smallsetminus \{y_0\}:\text{there is a rooted path in $G$ containing all $v\in S$}\},
\]
where a rooted path is any path starting from the root.
\end{definition}

Since \broom{} graphs are trees, it makes sense to consider their rooted path
complexes. We have the following:

\begin{corollary}
Let $I$ be a squarefree ideal.  Then $I$ is Gotzmann if and only if it
is the edge ideal of the rooted path complex of a labeled rooted \broom{}
graph.  The height of the graph is one less than the maximum degree of
a generator of $I$; that is, the height equals the regularity of $S/I$.
\end{corollary}

Observe that bristles attached to the root correspond to linear forms.

\todo{
All of this is true.  But I don't think it's necessary to announce
it.  The corollary is worth saying, but right now it's mentioned below
in counting, which should be sufficient.  

What do each of the graded components of a squarefree gotzmann ideal look like?
Is it always squarefree lex? Does the variable order stay consistent from one degree to the next?

Certainly each component isn't squarefree, but the squarefree monomials in each degree 
make a Gotzmann ideal in the KK ring. So it'll be interesting if they're not squarefree lexsegments.

The converse to the theorem holds by Murai and Hibi. Alternatively, the lemmas could be written as iffs; if $(x_i) \subseteq I$ then $I$ Gotz iff $I (S/(x_i))$ Gotz. That should do it.

\begin{corollary}
All Gotzmann squarefree monomial ideals in the polynomial rings are
squarefree lex ideals. ??
\begin{proof}
\end{proof}
\end{corollary}
}

We can use \broom{} graphs to count the number of Gotzmann squarefree
ideals of $S$.

\begin{lemma}
The number of labeled rooted \broom{} graphs of height $t$ with $m+1$
vertices labeled by
$\{1; x_{1},\dots, x_{m}\}$ is
\[
P(m,t)\left[(t+1)^{m-t}-t^{m-t}\right].
\]
\end{lemma}
\begin{proof}
There are $P(m,t)$ ways to label the handle; after doing so we must
assign each bristle to one handle vertex.  There are $(m-t)$ bristles,
so $(t+1)^{m-t}$ ways to do this.  However, we require that at least
one bristle be attached to the last vertex of the handle, so we must
subtract the $t^{m-t}$ ways to connect the bristles only to the other
vertices.
\end{proof}

\begin{lemma} The number of labeled rooted \broom{} graphs with $m+1$
  vertices labeled by $\{1;x_{1},\dots,x_{m}\}$ is
\[
\sum_{t<m} P(m,t)\left[(t+1)^{m-t}-t^{m-t}\right].
\]
\end{lemma}

\begin{proposition} The number of Gotzmann squarefree ideals of $S$ is 
\[
\sum_{m \leq n}\sum_{t<m}\binom{n}{m}P(m,t)\left[(t+1)^{m-t}-t^{m-t}\right].
\]
\end{proposition}

\begin{proposition}
The number of Gotzmann squarefree ideals of $S$ with regularity $t$ is 
\[
\sum_{m<n}\binom{n}{m}P(m,t)\left[(t+1)^{m-t}-t^{m-t}\right].
\]
\end{proposition}
}

Using Theorem \ref{classify}, it is possible, if difficult, to count
the Gotzmann squarefree ideals of $S$.  We begin by counting these
ideals up to symmetry.  (This is the same as counting the
``universally squarefree lex'' ideals:  the squarefree lex ideals
which are still squarefree lex in $S[y]$.)

\begin{proposition}  If $n\geq 2$, the following are all equal to $2^{n-2}$:
\begin{itemize}
\item[(i)] The number of ordered partitions of $n$ into an even number
  of summands.
\item[(ii)] The number of ordered partitions of $n$ into an odd number
  of summands.
\item[(iii)] The number of Gotzmann squarefree ideals which contain no
  linear forms and are not contained in any monomial subalgebra of $S$, up to a
  reordering of the variables.
\item[(iv)] The number of nonunit Gotzmann squarefree ideals which contain
  linear forms and are not contained in any monomial subalgebra of $S$, up to a
  reordering of the variables.
\item[(v)] The number of Gotzmann squarefree ideals which contain no
  linear forms and are contained in some monomial subalgebra of $S$, up to a
  reordering of the variables.
\item[(vi)] The number of Gotzmann squarefree ideals which contain 
  linear forms and are contained in some monomial subalgebra of $S$, up to a
  reordering of the variables.
\end{itemize}
In particular, there are $2^{n}$ nonunit Gotzmann squarefree ideals up to symmetry.
\end{proposition}
\begin{proof}
For (iii) through (vi), we describe a bijection to the ordered
partitions.  Given a partition, we partition the variables, in order,
into sets of the given sizes.  Using the notation of Theorem
\ref{classify}, we will alternate these sets between the supports of
the monomials $m_{i}$ and the sets
$\{x_{i_{j,1}},\dots,x_{i_{j,r_{j}}}\}$.  We begin with the monomial if we
are counting without linear forms, and with the set if we are counting
with linear forms (because $m_{1}=1$ in these cases).  If we are counting
ideals contained in a 
subalgebra, we do not use the last summand.  Note that the parity of the
partition is fixed in each case.

When $n=1$, the two nonunit Gotzmann squarefree ideals are $(0)$ and
$(x_{1})$.  
\end{proof}

We use the same idea of partitioning the variables and alternating
between monomials $m_{i}$ and sets $\{x_{i_{j,1}},\dots,x_{i_{j,r_{j}}}\}$
to count the Gotzmann squarefree ideals of $S$ without symmetry.  The
difficulty is that it is easy to overcount those ideals with
$r_{s}=1$.

Let $\mathsf G$ be the set of all squarefree Gotzmann ideals in polynomial
rings of the form $\field[x_1,\ldots, x_n]$ for some $n$.
We define a weight function $\omega:\mathsf G \to \mathbb N$ by $\omega(I) = n$ 
if $I \subset \field[x_1, \ldots, x_n]$.

We will show that the exponential generating function (e.g.f.) of $\mathsf G$ is
\[ 
	g(t) 
	= \sum_{I \in \mathsf G}  \frac{t^{\omega(I)}}{\omega(I)!} 
	= e^t \left(\frac{ 2(1 - t)}{2-e^t} +t\right).
\]
The coefficients this e.g.f.~count the number of squarefree Gotzmann ideals in polynomial
rings in $n$ variables for each value of $n$.

We begin with notation for ordered set partitions. In Proposition \ref{p:fullsupportgf}, we relate them to set 
$\mathsf H \subset \mathsf G$ of all squarefree Gotzmann ideals with full support (i.e. $I \in \mathsf H$ uses all $n$ variables for $n =\omega(I)$).
\begin{notation}[Ordered Set Partitions]
An ordered set partition of $[n] = \{1, \ldots, n\}$ is an ordered seqence $\sigma = (\sigma_1, \ldots, \sigma_k)$ of
sets $\sigma_i$ which partition $[n]$. Each $\sigma_i$ is called a \emph{block} of $\sigma$.

Let $\mathsf P_n$ be the set of ordered set partitions of $[n]$ and $\mathsf P$ be the union of all $\mathsf P_n$. 
On the set $\mathsf P$ we define a weight function $\nu: \mathsf P \to \mathbb N$ 
where $\nu(\sigma)$ is the number of elements that $\sigma$ partitions.

We will frequently use the e.g.f.~of $\mathsf P$ which counts the number of ordered set partitions of $[n]$:
\[ f(t) = \sum_{\sigma \in \mathsf P} \frac{t^{\nu(\sigma)}}{\nu(\sigma)!} = \frac{1}{2 - e^t}.
\]
\end{notation}
This e.g.f.~is entry A670 in the On-Line Encyclopedia of Integer Squences \cite{MR1992789}.

\begin{lemma} \label{l:lastblockgf}
Let $\mathsf P'$ be the set of ordered set partitions that have last blocks of size greater than one.
The e.g.f.~of $\mathsf P'$ with weight $\nu$ is $(1-t)/(2-e^t)$.
\begin{proof}
Let $\sigma \in \mathsf P \smallsetminus \mathsf P'$ be an ordered set partition with $\omega(\sigma) = n+1$.
Then the last block of $\sigma$ is the singleton $\{i\}$ for some $i=1, \ldots, n+1$.
Removing the last block from this partition gives a bijection between ordered set partitions ending in $\{i\}$ and ordered set partitions of a set of size $n$.
Thus, the exponential generating function of $\mathsf P \smallsetminus \mathsf P'$ is $t f(t) = \frac{t}{2-e^{t}}$
and hence the e.g.f.~of $\mathsf P'$ is $f(t) - tf(t) = \frac{1 - t}{2-e^t}$.
\end{proof}
\end{lemma}

The next proposition describes the relationship between ordered set partitions and the set $\mathsf H$ of Gotzmann squarefree ideals with full support.
\begin{proposition} \label{p:fullsupportgf}
The e.g.f.~of $\mathsf H$ with weight $\omega$ is 
\[ h(t) = \sum_{I \in \mathsf H}  \frac{t^{\omega(I)}}{\omega(I)!} = \frac{ 2(1 - t)}{2-e^t} +t.
\]
\begin{proof}
Every ideal in $\mathsf H$ is of the form
\[ 
m_1(x_{i_{1,1}}, \ldots, x_{i_{1,r_1}}) 
+ m_1 m_2 (x_{i_{2,1}}, \ldots, x_{i_{2, r_2}}) + \cdots 
+ m_1\cdots m_s(x_{i_{s,1}}, \ldots, x_{i_{s,r_s}})
\]
for some $m_j$ and $x_{i_{j,k}}$ all distinct. Let $\beta_{0,d}(I)$ be the number of generators of $I$ of degree $d$.
We partition $\mathsf H$ into five subsets $\mathsf H=\cup_{i=0}^4 \mathsf H_i$ where
\begin{align*}
	\mathsf H_0 &= \{ (x_1) \}, \\
	\mathsf H_1 &= \{ I \in H \mid \text{$I$ contains a linear form and $\beta_{0,\reg I}(I) = 1$ and $\reg I \neq 1$}\}, \\
	\mathsf H_2 &= \{ I \in H \mid \text{$I$ contains a linear form and $\beta_{0,\reg I}(I) > 1$}\}, \\
	\mathsf H_3 &= \{ I \in H \mid \text{$I$ does not contains a linear form and $\beta_{0,\reg I}(I) = 1$}\}, \text{ and} \\
	\mathsf H_4 &= \{ I \in H \mid \text{$I$ does not contains a linear form and $\beta_{0,\reg I}(I) > 1$}\}.
\end{align*}

Recall we use $\mathsf P'$ to denote the set of ordered set partitions whose last block is not a singleton.
There is a weight preserving bijection between $\mathsf P'$ and $\mathsf H_1 \cup \mathsf H_2$ given by
\[ (\sigma_1, \ldots, \sigma_k) \mapsto 
	\begin{cases} 
	(\sigma_1) + (\prod \sigma_2) (\sigma_3) + \cdots + (\prod_{i=1}^{(k-1)/2}\prod \sigma_{2i}) (\sigma_k)
		& k \text{ odd},\\
	(\sigma_1) + (\prod \sigma_2) (\sigma_3) + \cdots + (\prod_{i=1}^{k/2}\prod \sigma_{2i}) 
		& k \text{ even}.
	\end{cases}
\]

Similarly, there is a weight preserving bijection between $\mathsf P'$ and $\mathsf H_3 \cup \mathsf H_4$;
\[ (\sigma_1, \ldots, \sigma_k) \mapsto 
	\begin{cases} 
	(\prod \sigma_1) (\sigma_2) + (\prod \sigma_1)(\prod \sigma_3)(\sigma_4) + 
		\cdots + (\prod_{i=1}^{k/2}\prod \sigma_{2i-1})(\sigma_k) 
		& k \text{ even},\\
	(\prod \sigma_1) (\sigma_2) + (\prod \sigma_1)(\prod \sigma_3)(\sigma_4) + 
		\cdots + (\prod_{i=1}^{(k+1)/2}\prod \sigma_{2i-1})
		& k \text{ odd}.
	\end{cases}
\]

The desired formula follows from Lemma \ref{l:lastblockgf}.
\end{proof}
\end{proposition}

\begin{corollary} The exponential generating function for $\mathsf G$, the set of all Gotzmann squarefree monomial ideals, is
\[ 
	g(t) 
	= \sum_{I \in \mathsf G} \frac{t^{\omega(I)}}{\omega(I)!} 
	= e^t \left(\frac{2(1 - t)}{2 - e^t} + t\right)
\]
\begin{proof}
For each Gotzmann squarefree monomial ideal with full support in a polynomial ring over $k$ variables 
there are $\binom{n}{k}$ Gotzmann squarefree monomial ideals in a polynomial ring over $n$ variables with support of size $k$. 
Thus, we apply the inverse binomial transform to the previous proposition (i.e. multiply the e.g.f.~by $e^t$).
\end{proof}
\end{corollary}

From this generating function, one can extract the number of squarefree Gotzmann ideals in a polynomial ring in $n$ generators.
For the first few values of $n$, these numbers are $2,3,6,19,96,669$ (i.e. for $n=0,1,\ldots,5$).

\section{Gotzmann Ideals of the Squarefree Ring}

The problem of classifying all Gotzmann monomial ideals of the
squarefree ring $R$ turns out to be much more difficult. We might
hope to prove some squarefree analog of Lemma \ref{linform}; then,
arguing as in the previous section, we would be able to prove that
Gotzmann ideals of $R$ are lex segments (if generated in one degree) or
initial segments in a lexlike tower (see \cite{MR2253664}) in general.
Unfortunately such an approach is doomed to fail, as the following
examples show.

\begin{example}
The ideal $I=(ab,ac,bd,cd)$ is Gotzmann in $R$ but is not lex.
\end{example}

The ideal $I$ above is (up to symmetry) the only monomial
Gotzmann ideal of $\field[a,b,c,d]/(a^{2},b^{2},c^{2},d^{2})$ which is not
lex in some order.  Thus we might hope that it is the only such ideal,
or at least is the first instance of a one-parameter family of
exceptions.  This hope is dashed as well as soon as we add a fifth
variable.

\begin{example}
The ideal $I=(abc,abd,abe,acd,ace,bcd,bce)$ is Gotzmann in $R$ but
is not lex.
\end{example}

Since the Alexander duals of lex ideals are lex, we might hope that
the Alexander duals of Gotzmann ideals are Gotzmann.  However, the
duals of the two examples above are not Gotzmann.  We will see in
Theorem \ref{dualSurprise} that a Gotzmann ideal has Gotzmann dual if
and only if it is in some sense morally lex.

Throughout the section, all ideals will be monomial ideals of $R$.
Since we no longer work with the polynomial ring, we can dispense with the
notation $I^{\sqf}$ to indicate that an ideal lives in $R$, and will
simply write $I$, $J$, etc.  Many of our arguments are technical, so
for ease of notation we work mostly with monomial vector spaces rather
than ideals.  Recall that a vector space $V\subset R_{d}$ is Gotzmann if
$|\mathbf{m}_{1}V|$ is minimal given $|V|$ and $d$, and that an ideal $I$
is Gotzmann if and only if $|I_{d}|$ is Gotzmann for all $d$.  

\subsection{Decomposing Gotzmann Ideals of $R$}

In this section we show every Gotzmann monomial vector space $V \subseteq R_d$
can be decomposed as the direct sum of two monomial vector spaces which are
Gotzmann in a squarefree ring with one fewer generator. This decomposition relates to the
operation of compression (see \cite{MR2231127} or
\cite{MR2409180}). We begin by recalling the necessary notation.

Given a (fixed) variable $x_i$, let $\mathbf{n} = (x_1, \ldots, \hat x_i, \ldots, x_n)$
be the maximal ideal in $Q = R/(x_i)$ which is a squarefree ring on $n-1$ variables.

\begin{definition}[$x_i$-decomposition]
Let $V \subseteq R_d$ be a monomial vector space and fix a variable $x_i$.
The monomial basis of $V$ can be partitioned as $A \cup B$ where
$A$ contains the monomials divisible by $x_i$ and $B$ contains those not divisible by $x_i$.

Let $\Vhatxi{}$ be the monomial vector space spanned by $B$ and let $\Vxi$ be
the monomial vector space spanned by $\{ m \mid x_i m \in A\}$. 
We write $V$ as the direct sum
\[	V = \Vhatxi{} \oplus x_i \Vxi
\]
which we call the \emph{$x_i$-decomposition} of $V$.
\end{definition}

We view the monomial vector spaces $\Vhatxi{}$ and $\Vxi$
as subspaces of $Q_d$ and $Q_{d-1}$ respectively.

\begin{definition}[$x_i$-compression]
Let $V = \Vhatxi{} \oplus x_i \Vxi$ be the $x_i$-decomposition of
the monomial vector space $V$. Let $\Lhatxi{}$ and $\Lxi{}$ 
be the squarefree lex-segments in $Q$ with the same degrees and dimensions as $\Vhatxi{}$ and $\Vxi$. The \emph{$x_i$-compression} of $V$ is the monomial vector space
\[ T= \Lhatxi{} \oplus x_i \Lxi{}.
\]
\end{definition}

We recall the following important fact about compressions from \cite{MR2231127}:
\begin{proposition}[\cite{MR2231127}]\label{compressionIsIdeal}
If $T$ is the $x_i$-compression of the monomial vector space $V \subseteq R_d$,
then 
\[	|\mathbf{m}_1 T|\leq |\mathbf{m}_1 V|.
\]
\end{proposition}

\begin{lemma} \label{decompGrowth}
If $V = \Vhatxi{} \oplus \Vxi$ then the $x_i$-decomposition of
$\mathbf{m}_{1}V$ is
\[ \mathbf{m}_1 V = \mathbf{n}_1 \Vhatxi{} \oplus x_i (\Vhatxi{}+\mathbf{n}_1 \Vxi).
\]
\begin{proof}
Since $x_{i}^{2}=0$, we have $\mathbf{m}_1 (x_i \Vxi) = \mathbf{n}_1(x_i \Vxi)$.
Thus,
\begin{align*}
	 \mathbf{m}_1 V 
	&= \mathbf{m}_1 (\Vhatxi{} + x_i \Vxi) \\
	&= \mathbf{n}_1 \Vhatxi{} + x_i \Vhatxi{} + x_i \mathbf{n}_1 \Vxi  \\
	&= \mathbf{n}_1 \Vhatxi{} \oplus x_i (\Vhatxi{} + \mathbf{n}_1 \Vxi). 
\end{align*}
This sum is direct since the second summand is contained in $(x_{i})$
while the first summand is not.
\end{proof}
\end{lemma}

\begin{proposition}
Let $V \subseteq R_d$ be a Gotzmann monomial vector space and let
$V = \Vhatxi{} \oplus x_i \Vxi$ be its $x_i$-decomposition.
Then $\Vhatxi{}$ is Gotzmann in $Q$.
\begin{proof}
Let $L$ be the $x_i$-compression of $V$. As $V$ is Gotzmann $|\mathbf{m}_1 V| \leq |\mathbf{m}_1 L|$
and so $|\mathbf{m}_1 V| = |\mathbf{m}_1 L|$ by Proposition
\ref{compressionIsIdeal}.  

Thus we have
\begin{equation*} 
	  |\mathbf{n}_1 \Vhatxi{}| + |\Vhatxi{}+\mathbf{n}_1 \Vxi| 
	= |\mathbf{n}_1 \Lhatxi{}| + |\Lhatxi{}+\mathbf{n}_1 \Lxi{}| \tag{$\star$} \label{growthEquality}
\end{equation*}
from the previous lemma.

Since $\Lxi{}$ and $\mathbf{n}_1 \Lhatxi{} $ are lex segments of the same degree, it follows that one contains
in the other. If $\mathbf{n}_1 \Lxi{} \subseteq \Lhatxi{}$ then
\[
	|\Lhatxi{}+\mathbf{n}_1 \Lxi{}| 
	= | \Lhatxi{} | 
	= | \Vhatxi{} |
	\leq | \Vhatxi{} + \mathbf{n}_1 \Vxi|.
\]
Similarly, if $\Lhatxi{} \subseteq \mathbf{n}_1 \Lxi{}$ then
\[
	|\Lhatxi{}+\mathbf{n}_1 \Lxi{}| 
	= |\mathbf{n}_1 \Lxi{}|  
	\leq |\mathbf{n}_1 \Vxi| 
	\leq | \Vhatxi{} + \mathbf{n}_1 \Vxi|.
\]

In both cases $|\Lhatxi{}+\mathbf{n}_1 \Lxi{}| \leq | \Vhatxi{} + \mathbf{n}_1 \Vxi|$.
From the equality above we see that $|\mathbf{n}_1 \Vhatxi{}| \leq |\mathbf{n}_1 \Lhatxi{}|$ and hence $\Vhatxi{}$ is Gotzmann by Proposition \ref{gotzByLexSeg}.
\end{proof}
\end{proposition}

\idiot{The old proof is as follows:

By construction, $|\mathbf{n}_1 \Vhatxi{}| \geq |\mathbf{n}_1 \Lhatxi{}|$, and so
\begin{align*}
|\mathbf{n}_1 \Vhatxi{} + \Vxi | & \leq |\mathbf{n}_1 \Lhatxi{} + \Lxi{} |\\
|\mathbf{n}_1 \Vhatxi{}| + |\Vxi| - |\mathbf{n}_1 \Vhatxi{} \cap  \Vxi | & \leq |\mathbf{n}_1 \Lhatxi{} |+| \Lxi{} |-|\mathbf{n}_1 \Lhatxi{} \cap \Lxi{} |\\
|\mathbf{n}_1 \Vhatxi{}| - |\mathbf{n}_1 \Vhatxi{} \cap  \Vxi | & \leq |\mathbf{n}_1 \Lhatxi{} |-|\mathbf{n}_1 \Lhatxi{} \cap  \Lxi{} |.
\end{align*}

Since $\Lxi{}$ and $\mathbf{n}_1 \Lhatxi{} $ are lex, it follows that one is contained
in the other.  If $\mathbf{n}_1 \Lhatxi{} \subset  \Lxi{} $, then the right-hand side
above is zero, and we conclude that $\mathbf{n}_1 \Vhatxi{} \subset  \Vxi $ and
$|\mathbf{n}_1  \Vxi |=|\mathbf{n}_1  \Lxi{} |$.  If 
$ \Lxi{} \subset \mathbf{n}_1 \Lhatxi{} $, then the right-hand side above is
$|\mathbf{n}_1 \Lhatxi{} |-| \Lxi{} |$ and we have

\begin{align*}
|\mathbf{n}_1 \Vhatxi{} |-| \Lxi{} |&= |\mathbf{n}_1 \Vhatxi{} |-| \Vxi |\\
&\leq |\mathbf{n}_1 \Vhatxi{} |-|\mathbf{n}_1 \Vhatxi{}  \cap  \Vxi |\\
&\leq |\mathbf{n}_1 \Lhatxi{} |-|\mathbf{n}_1 \Lhatxi{}  \cap  \Lxi{} |\\
&=|\mathbf{n}_1 \Lhatxi{} |-| \Lxi{} |,
\end{align*}
i.e., $|\mathbf{n}_1 \Vhatxi{}| \leq |\mathbf{n}_1 \Lhatxi{} |$.  But $\Lhatxi{} $ is lex, so we conclude that
these are equal and in particular all the inequalities above are
equalities.  
}

\begin{lemma}\label{almostThere}
Let $V$ be Gotzmann in $R$ with $x_{i}$-decomposition $V = \Vhatxi{} \oplus x_i \Vxi$
and let $L = \Lhatxi{} \oplus x_i \Lxi{}$ be its $x_i$-compression.
Then either $\Vxi$ is Gotzmann in $Q$ or $\mathbf{n}_1 \Lxi{} \subset
\Lhatxi{}$. 
\begin{proof}
We know from the previous proposition that $\Vhatxi{}$ is Gotzmann in $Q$ and hence
$|\mathbf{n}_1 \Vhatxi{}| = |\mathbf{n}_1 \Lhatxi{}|$. Thus, the equality \eqref{growthEquality}
gives
\[
	  |\Vhatxi{}+\mathbf{n}_1 \Vxi| 
	= |\Lhatxi{}+\mathbf{n}_1 \Lxi{}|.
\]
If $\mathbf{n}_1 \Lxi{} \not \subset \Lhatxi{}$ then $\Lhatxi{} \subseteq \mathbf{n}_1 \Lxi{}$ as they are both lex segments. Thus 
\[ |\mathbf{n}_1 \Vxi | 
	\leq |\Vhatxi{} + \mathbf{n}_1 \Vxi| 
	= |\Lhatxi{}+\mathbf{n}_1 \Lxi{}|
	= |\mathbf{n}_1 \Lxi{}|
\]
which proves that $\Vxi$ is Gotzmann.
\end{proof}
\end{lemma}

If $\mathbf{n}_1 \Lxi{} \subset \Lhatxi{}$, then $\Vxi$ need not be Gotzmann. For
example, 
\[ V=\spn_{\field}\{abc, abd, acd, bcd, bce, bde, cde\}
\] 
is Gotzmann in $R = \field[a,b,c,d,e]/(a^2, \ldots, e^2)$, 
but $V_{1}=\spn_{\field}\{bc, bd, cd\}$ from the $a$-decomposition of $V$ is not Gotzmann in $Q = R/(a)$.

However, we will see that it is always possible to choose $x_i$ such that $\Vxi$ is
Gotzmann.

\idiot{I think this becomes unnecessary
\begin{lemma}
For a given variable $x_i$, let $V = \Vhatxi{} \oplus x_i \Vxi$ be the $x_i$-decomposition of the monomial vector space $V \subseteq R$ and let $L = \Lhatxi{} \oplus x_i \Lxi{}$ be its $x_i$-compression. 
If $x_i$ is a variable with $\Lhatxi{} \subseteq \mathbf{n}_1 \Lxi{}$, then any variable $x_j$ with $|V_{x_j}|\geq |\Vxi|$ also satisfies $L_{\hat x_j} \subseteq \mathbf{n'}_1 L_{x_j}$ where $\mathbf{n'} = (x_1, \ldots, \hat x_j, \ldots, x_n) \subset R/(x_j)$.
\begin{proof}
Assume $x_i$ and $x_j$ are as given in the statement of the lemma.
As $|V_{x_j}|\geq |\Vxi|$, we have
$|L_{x_j}| \geq |\Lxi{}|$ since 
$|L_{x_j}| = |V_{x_j}|$ and $|\Lxi{}| = |\Vxi|$.

We also have $|L_{\hat x_j}| \leq |\Lhatxi{}|$
since $|L_{\hat x_j}| = |V_{\hat x_j}|$ and 
$|V_{\hat x_j}| = |V|- |V_{x_j}|$ and similarly for $x_i$.

As $|\Lxi{}| \leq |L_{x_j}|$, it is clear that
$|\mathbf{n}_1 \Lxi{}| \leq |\mathbf{n'}_1 L_{x_j}|$
as we are simply comparing the growth of two nested lex segments in $n-1$ variables 
(albeit, differently named variables). Thus,
\[ 
|L_{\hat x_j}|
\leq |\Lhatxi{}|
\leq |\mathbf{n}_1 \Lxi{}|
\leq |\mathbf{n'}_1 L_{x_j}|
\]
Since $L_{\hat x_j}$ and $\mathbf{n'}_1 L_{x_j}$ are lex segments, we get the desired containment.
\end{proof}
\end{lemma}
}

\begin{lemma}\label{exchange}
Let $V$ be Gotzmann with $x_{i}$-decomposition $V=\Vone\oplus
x_{i}\Vxi$ and compression $L=\Lone\oplus x_{i}\Lxi$.  If
$\mathbf{n}_{1}\Lxi\subseteq\Lone$, then $V$ satisfies the 
property:  
\begin{quotation}
Let $m\in V$ be a monomial such that $x_{i}$ divides $m$, and let
$x_{j}$ be any variable not dividing $m$.  Then
$\frac{x_{j}}{x_{i}}m\in V$.
\end{quotation}
\end{lemma}
\begin{proof}  
Applying \eqref{growthEquality},
  we have
  $|\mathbf{n}_{1}\Vxi+\Vone|=|\mathbf{n}_{1}\Lxi+\Lone|=|\Lone|=|\Vone|$, i.e.,
 $\mathbf{n}_{1}\Vxi\subseteq\Vone$.  The desired property
  follows.
\end{proof}

\idiot{Actually, this isn't necessary either.
\begin{proposition} Let $V$ be a Gotzmann monomial vector space $V$. 
Then there exists some variable $x_i$ such that $\Lhatxi{} \subseteq
\mathbf{n}_1 \Lxi{}$  
where $\Lhatxi{} \oplus x_i \Lxi{}$ is the $x_i$-compression of $V$.
\end{proposition}
\begin{proof}
\idiot{This is the old proof.

By the previous lemma it suffices to check if $\Lhatxi{} \oplus x_i \Lxi{}$ when $x_i$ is the variable for which $|\Vxi|$ is
largest. That is, $x_i$ divides the largest possible number of monomial generators of $V$.

Assume for a contradiction that $\Lhatxi{} \not \subseteq \mathbf{n}_1 \Lxi{}$. As these
are two lex segments in the same degree we have $\mathbf{n}_1 \Lxi{} \subset \Lhatxi{}$.

By Lemma \ref{decompGrowth}, 
\[ | \mathbf{m}_1 V| = |\mathbf{n}_1 \Vhatxi{}| + |\Vhatxi{} + \mathbf{n}_1 \Vxi |
\]
and
\[ | \mathbf{m}_1 L| = |\mathbf{n}_1 \Lhatxi{}| + |\Lhatxi{} + \mathbf{n}_1 \Lxi{} |.
\]

Since $V$ and $\Vhatxi{}$ are Gotzmann, we have $|\mathbf{m}_1 V| = | \mathbf{m}_1 L|$
and $|\mathbf{n}_1 \Vhatxi{}| = |\mathbf{n}_1 \Lhatxi{}|$ which gives
\begin{align*}
	|\Vhatxi{} + \mathbf{n}_1 \Vxi |
	&=|\Lhatxi{} + \mathbf{n}_1 \Lxi{} |\\
	&=|\Lhatxi{}|
\end{align*}
as $\mathbf{n}_1 \Lxi{} \subset \Lhatxi{}$. Since $|\Vhatxi{}| = |\Lhatxi{}|$
we get $\mathbf{n}_1 \Vxi \subseteq \Vhatxi{}$.

Let $x_j$ be a variable with $|V_{x_j}| < |\Vxi|$ and consider the monomials in $\Vxi$ 
which are not divisible by $x_j$. For each such monomial $m$, the monomial $mx_j$ is
in $\Vhatxi{}$ as $\mathbf{n}_1 \Vxi \subseteq \Vhatxi{}$. Therefore, there are at
least $|\Vxi|$ monomials in $V$ which divisible by $x_j$. This contradicts our
choice of $x_j$.
}
\idiot{
The following is the old proof. It doesn't work because it assumes that $T$ (the compression) doesn't depend on which variable we're compressing with respect to. I've added the condition that 
not all $|\Vxi|$ are equal. Is there a way around this?

Suppose not.  Then in particular $V$ is nonzero and $\mathbf{n}_1 \Lxi{}\subset
\Lhatxi{}$ for all $i$. 

Let $u$ be any monomial in $L$ and $x_j$ be a variable dividing $u$ and $x_k$ be a variable
not dividing $u$. As $\mathbf{n}_1 L_{x_j} \subseteq L_{\hat x_j}$ and $u/x_j \in L_{x_j}$,
we have $x_k u /x_j \in L_{\hat x_j}$.

Let $u$ be any
monomial of $T$, and $v$ another monomial of degree $d$.  We will show
that $v\in T$.  Let $\{a_{1},\dots, a_{r}\}$ be the monomials dividing
$u$ but not $v$ and $\{b_{1},\dots, b_{r}\}$ be the monomials dividing
$v$ but not $u$.  Setting $a=a_{1}$, we have $\frac{u}{a}\in T_{a}$,
so $u\frac{b_{1}}{a_{1}}\in \mathbf{n}_1 T_{a}\subset T_{1}\subset T$.  Inducting on $r$,
we have $v\in T$.  Since $v$ was arbitrary, it follows that
$I=T=X^{d}$.  Thus $\mathbf{n}_1 T_{a}=T_{1}=\mathbf{n}_1 ^{d}$.
}

\end{proof}

\todo{
Remove the condition that all $|\Vxi|$ are equal or show that it's needed.
}

Thus we have the following theorem:
}

\begin{theorem}\label{decomposeTheorem}
Suppose $V\subset R_{d}$ is a Gotzmann monomial vector space.
Then $x_i$ may be chosen so that both summands $\Vxi$ and 
$\Vhatxi{}$ of the $x_i$-decomposition of $V$ are Gotzmann 
in $Q$ and $\Vhatxi{}\subseteq \mathbf{n}_1 \Vxi{}$.
\end{theorem}
\begin{proof}
Suppose that $x_{i}$ cannot be chosen so that the summands $\Lxi$ and
$\Lone$ of the $x_{i}$-compression satisfy $\Lone\subseteq
\mathbf{n}_{1}\Lxi$.   
Then Lemma \ref{exchange} applies for all
$x_{i}$, so $V$ satisfies the property:
\begin{quotation}
Let $m\in V$ be a monomial, and suppose that $x_{i}$ divides $m$ and
$x_{j}$ does not.  Then $\frac{x_{j}}{x_{i}}m\in V$ as well.
\end{quotation}
The only subspaces of $R_{d}$ satisfying this property are $(0)$ and
$R_{d}$.  In either case, we have  $\Lone\subseteq
\mathbf{n}_{1}\Lxi$ for any $x_{i}$.   

Thus, $x_{i}$ may be chosen such that  $\Lone\subseteq
\mathbf{n}_{1}\Lxi$.  Then by Lemma \ref{almostThere}
$\Vxi$ and $\Vone$ are Gotzmann
in $Q$.  Applying \eqref{growthEquality}, we have
$|\Vone+\mathbf{n}_{1}\Vxi|=|\Lone+\mathbf{n}_{1}\Lxi|=|\mathbf{n}_{1}\Lxi|=|\mathbf{n}_{1}\Vxi|$,
i.e., $\Vhatxi{}\subseteq \mathbf{n}_1 \Vxi{}$.
\end{proof}

In fact, the obvious choice of variable works:
\begin{lemma}\label{topVarLemma}
Suppose $V\subset R_{d}$ is a Gotzmann monomial vector space, and let
$x_{i}$ be such that $|V\cap (x_{i})|$ is maximal.  Let
$V=\Vone\oplus x_{i}\Vxi$ be the $x_{i}$-decomposition of $V$.  Then
$\Vone$ and $\Vxi$ are both Gotzmann in $Q$ and $\Vhatxi{}\subseteq
\mathbf{n}_1 \Vxi{}$.  
\end{lemma}
\begin{proof}
Let $\Lone$ and $\Lxi$ be the lexifications in $Q$ of $\Vone$ and
$\Vxi$, respectively.

By Theorem \ref{decomposeTheorem}, there exists a variable $x_{j}$
such that we may decompose $V=\Wone\oplus x_{j}\Wxj$ with both $\Wone$
and $\Wxj$ Gotzmann in $Q$ and $\Wone\subseteq\mathbf{n}_{1}\Wxj$.  

We have 
\[
|\Lone|\leq |\Wone|\leq |\mathbf{n}_{1}\Wxj|\leq |\mathbf{n}_{1}\Lxi|,
\]
the first inequality by construction, the second by Theorem
\ref{decomposeTheorem}, and the third because $|\Wxj|\leq |\Lxi|$
and both are Gotzmann.  By Lemma \ref{almostThere}, $\Vxi$ is
Gotzmann.  Applying \eqref{growthEquality} again, we obtain $\Vhatxi{}\subseteq
\mathbf{n}_1 \Vxi{}$.    
\end{proof}

Unfortunately the converse to Theorem \ref{decomposeTheorem} does not
hold in general.  For example, let 
$V=\spn_{\field} \{ab,ac,bc\}$ in $R=\field[a,b,c,d]/(a^2, \ldots, d^2)$. 
Then $V$ is not Gotzmann in $R$ but, decomposing with respect to $a$,
$V_{0}=\spn_{\field} \{bc\}$ 
and $V_{1}= \spn_{\field} \{ b,c \}$ 
are both Gotzmann in $Q=\field[b,c,d]/(b^2,c^2,d^2)$.

\todo{
Prepare the next theorem with the necessary information about Macaulay represenations.
Either that or find a way to use lex segments to prove it.
}

However, we can prove the following partial converse.

\begin{theorem} \label{reconstructTheorem}
Let $\Vhatxi{}$ and $\Vxi$ be Gotzmann monomial vector spaces in $Q$ 
with $\Vhatxi{} = \mathbf{n}_1  \Vxi$. Then
$V = \Vhatxi{} \oplus \mathbf{n}_1  \Vxi$ is Gotzmann in $R$. 
\begin{proof}
Choose any lex order in which $x_{i}$ comes last, and let $L=\Lone+
x_{i}\Lxi$ be the $x_{i}$-compression of $V$.  
We have
$|\mathbf{m}_{1}V|=|\mathbf{n}_{1}\Vone|+|\Vone+\mathbf{n}_{1}\Vxi|=|\mathbf{n}_{1}\Vone|+|\Vone|=|\mathbf{n}_{1}\Lone|+|\Lone|=|\mathbf{n}_{1}\Lone|+|\Lone+\mathbf{n}_{1}\Lxi|=|\mathbf{m}_{1}L|$.
Thus, it suffices to show that $L$ is lex.

Indeed, suppose that $u\in L$ and $v$ is a monomial of the same degree
which precedes $u$ in the lex order.  If both or neither of $u,v$ are
divisible by $x_{i}$, then clearly $v\in L$.  Now suppose that $u$ is
divisible by $x_{i}$ but $v$ is not.  Then we may write $u=u'x_{i}$.
By construction, $v$ precedes $u'$ in the (ungraded) lex order.  Let
$v'=\frac{v}{x_{j}}$, where $x_{j}$ is the lex-last variable dividing
$v$.  Then $v'$ precedes $u'$ in the lex order as well, so $u'\in
\Lxi$ implies $v'\in\Lxi$ and in particular $v\in
\mathbf{n}_{1}\Lxi=\Lone$.  A similar argument shows that $v\in L$ if
$v$ is divisible by $x_{i}$ but $u$ is not.
\idiot{This is the old proof.

Let the $d-1$-th Macaulay representation of $|Q_{d-1}/\Vxi|$ be 
\[ |Q_{d-1}/\Vxi| = \binom{a_{d-1}}{d-1} + \cdots + \binom{a_1}{1}
\]
where $a_{d-1} > a_{d-2} > \cdots >  a_1 \geq 0$. Then the $d$-th Macaulay representation
of $|Q_d/\Vhatxi{}|$ is
\[ |Q_{d}/\Vhatxi{}| = \binom{a_{d-1}}{d} + \cdots + \binom{a_1}{2}
\]
as $\Vxi$ is Gotzmann, $\mathbf{n}_1  \Vxi = \Vhatxi{}$ and by the Kruskal-Katona theorem.

Since $|R_d| = \binom{n}{d} = \binom{n-1}{d} + \binom{n-1}{d-1} = |Q_d| + |Q_{d-1}|$,
we can compute the dimension of $R_d/V$ as
\begin{align*} 
	|R_d / V| 
	&= |R_d| - |V| \\
	&= |Q_d| + |Q_{d-1}| - |\Vhatxi{}| - |\Vxi| \\
	&= |Q_d/\Vhatxi{}| + |Q_{d-1}/\Vxi|.
\end{align*}

Thus, $|R_d /V|$ has $d$-th Macaulay representation
\begin{align*} 
	|R_d / V| 
	&= \binom{a_{d-1}}{d} + \cdots + \binom{a_1}{2} \\
	&\qquad + \binom{a_{d-1}}{d-1} + \cdots + \binom{a_1}{1} \\
	&= \binom{a_{d-1}+1}{d} + \cdots + \binom{a_1 +1}{2}.
\end{align*}

The monomial vector space $V$ is Gotzmann if 
\begin{align*}
	|R_{d+1} / \mathbf{m}_1 V| 
	&= |R_d/V|^{(d)} \\
	&= \binom{a_{d-1}+1}{d+1} + \cdots + \binom{a_1 + 1}{3}
\end{align*}
and this is what will be shown.

As $\mathbf{n}_1 \Vxi = \Vhatxi{}$ we have $\mathbf{m}_1 V = \mathbf{n}_1 \Vhatxi{} \oplus x_i (\Vhatxi{} + \mathbf{n}_1 \Vxi) = \mathbf{n}_1 \Vhatxi{} \oplus x_i \Vhatxi{}$.
\todo{
Is $\Vhatxi{}$ gotzmann in $R$ if it's Gotzmann in $Q$? Probably not.
If it were, $\Vhatxi{} + x_i J$ where $J \subseteq \Vxi$ is arbitrary would probably be Gotzmann.

Hmm. I don't remember what I was getting at here. It's probably not important.
}

Therefore 
\begin{align*}
	|R_{d+1}/ \mathbf{m}_1 V| 
	&= |Q_{d+1} / \mathbf{n}_1 \Vhatxi{}|  + |Q_d / \Vhatxi{}|\\
	&= \binom{a_{d-1}}{d+1} + \cdots + \binom{a_1}{3} \\
	&\qquad + \binom{a_{d-1}}{d} + \cdots + \binom{a_1}{2}\\
	&= \binom{a_{d-1}+1}{d+1} + \cdots + \binom{a_1+1}{3} 
\end{align*}
and so $I$ is Gotzmann by the Kruksal-Katona theorem.
}
\end{proof}
\end{theorem}

\begin{example} 
Consider the Gotzmann vector space 
\[ V_{1} = \spn_{\field} \{ ab,bc,cd,ad\}
\]
in $Q  = \field[a,b,c,d]/(a^2, \ldots, d^2)$. Let $V_{0} = \mathbf{n}_1 V_{1}$:
\[ V_{0} = \spn_{\field}  \{
abc,
abd,
acd,
bcd
\}.
\]
In $R = \field{}[a,b,c,d,e]/(a^2, \ldots, e^2)$, the monomial vector space 
$V = V_{0} + e V_{1}$ is
Gotzmann but is not lex with respect to any order of the variables.
\end{example}

\subsection{Alexander Duality}
Recall that for a monomial vector space $V\subseteq R_{d}$, the
\emph{Alexander dual} of $V$ is the subspace $V^{\vee}\subset R_{n-d}$
spanned by the monomials $\{\frac{\mathbf{x}}{m}:m\not\in V\}$ where
$\mathbf{x}$ is the product of all the variables.  For a
monomial ideal $I\subset R$, the Alexander dual is $I^{\vee}=\oplus
(I_{d})^{\vee}$.  This duality corresponds to topological Alexander
duality under the Stanley-Reisner correspondence, and turns out to
have many nice algebraic properties.  For example, duality turns generators into
associated primes, and the duals of lex or Borel ideals are always lex
or Borel, respectively.  Thus, we would like to understand ideals
whose duals are Gotzmann.
  
\begin{definition}
We say that a monomial vector space $V$ is \emph{\gdual{}} if $V^{\vee}$
is Gotzmann.
\end{definition}

\begin{theorem}\label{dualDecompose}
Let $V$ be \gdual{} in $R$.  Then $x_{i}$ may be chosen so that both summands
$\Vone$ and $\Vxi$ of the $x_{i}$-decomposition are \gdual{} in $Q$, and 
$(\Vone: \mathbf{n}_{1}) \subseteq \Vxi$.
\end{theorem}
\begin{proof}  
Let $W=V^{\vee}$.  Then Theorem \ref{decomposeTheorem} applies to $W$,
so we may choose $x_{i}$ such that $\Wone$ and $\Wxi$ are Gotzmann in
$Q$ and  $\Wone\subseteq \mathbf{n}_{1}\Wxi$.

We compute $\Vone=(\Wxi)^{\vee}$ and $\Vxi=(\Wone)^{\vee}$.  In particular, $\Vone$ and
$\Vxi$ are \gdual{}.
Finally, suppose that $m\in (\Vone:\mathbf{n}_{1})$.  We will show
that $m\in\Vxi$.  By construction, $mx_{j}\in \Vone$ for all
$x_{j}\neq x_{i}$ and not dividing $m$, so
$\frac{\mathbf{x}}{mx_{j}}\not\in \Wxi$ for any such $x_{j}$.  Hence
$\frac{\mathbf{x}}{m}\not\in\mathbf{n}_{1}\Wxi$.  Since $\Wone\subseteq
\mathbf{n}_{1}\Wxi$, we have $\frac{\mathbf{x}}{m}\not\in\Wone$.  Thus
$m\in \Vxi$, as desired.
\end{proof}

Thus, any recursive enumeration of \gdual{} ideals should look similar
to any recursive enumeration of Gotzmann ideals.  However, they will
not be identical.  In fact, ideals which are simultaneously Gotzmann
and \gdual{} are quite rare, as the next theorem shows.

\idiot{This doesn't seem necessary after all.

\begin{proposition}Duality commutes with compression.  That is, if
  $T$ is the $x_{i}$-compression of $V$ then $T^{\vee}$ is the
  $x_{i}$-compression of $V^{\vee}$.
\end{proposition}
\begin{proof}
Write $V=\Vone+x_{i}\Vxi$ and $T=\Lone+x_{i}\Lxi$.  Then
$V^{\vee}=(\Vxi)^{\vee}+x_{i}(\Vone)^{\vee}$, and similarly for
$T^{\vee}$ (with the duals on the right-hand side taken in $Q$).  Thus
$|\Lone^{\vee}|
=|\Vone^{\vee}|$ and $|\Lxi^{\vee}|=|\Vxi^{\vee}|$ as necessary.
\end{proof}
}

\begin{theorem}\label{dualSurprise}
Suppose that $V\subset R_{d}$ is both Gotzmann and \gdual{}.  Then $V$
is lex in some order.
\end{theorem}
\begin{proof}
Suppose not.  Then there exists a counterexample $V\subset R_{d}$
where $R=\field[x_{1},\dots,x_{n}]/(x_{1}^{2},\dots,x_{n}^{2})$ with $n$
minimal.  
Let $x_{i}$ be such that $|V\cap (x_{i})|$ is maximal.  Then
$|V^{\vee}\cap (x_{i})|$ is maximal as well, and Lemma
\ref{topVarLemma} applies to both $V$ and $V^{\vee}$.  
Thus
$\Vone$ and $\Vxi$ are both Gotzmann and \gdual{}, so, by the
minimality of $n$, both are lex in $Q$.  Since $V$ is not lex, 
we have $\Vone\neq 0$ and $\Vxi\neq Q_{d-1}$.
Since $\Vone\neq 0$, we have $\mathbf{m}_{n-d-1}V=R_{n-1}$.  Thus the
lexification of $V$ (in any order where $x_{i}$ comes first) must
contain at least one monomial not divisible  
by $x_{i}$.  Similarly, the lexification of $V^{\vee}$ must contain at
least one monomial not divisible by $x_{i}$.   Thus, if $L$ and
$L^{\vee}$ are the lexifications of $V$ and $V^{\vee}$, respectively,
\idiot{$L$ and $L^{\vee}$ are also duals of each other.  This looks
  like a pain to prove and isn't necessary.}
we have
\begin{align*}
|L|+|L^{\vee}|&\geq  |Q_{d-1}|+1 +|Q_{n-d-1}|+1\\
&\gneqq |Q_{d-1}|+|Q_{d}|\\
&= |R_{d}|.
\end{align*}
On the other hand, $|L|+|L^{\vee}|=|V|+|V^{\vee}|=|R_{d}|$.  Thus,
such a minimal counterexample cannot exist.
\end{proof}

Note that Theorem \ref{dualSurprise} is not a theorem about ideals.
If an ideal $I$ is both Gotzmann and \gdual{}, then Theorem
\ref{dualSurprise} guarantees that every degreewise component $I_{d}$
is lex in some order, but does not guarantee a consistent order.  For
example, the ideal $I=(bc,abd,abe,acd,ace,ade)\subset \field[a,b,c,d,e]/(a^{2},b^{2},c^{2},d^{2},e^{2})$ is
Gotzmann and 
\gdual{}, but is not lex in any order.  The component $I_{p}$ is lex
with respect to the order $a>b>c>d>e$ for $p\neq 2$, and with respect
to the order $b>c>a>d>e$ for $p<3$, but no lex order works in both
degrees two and three.

\todo{delete this:
\subsection{Compressed Gotzmann ideals}
\todo{Prove something worthwhile here.}

We would like to understand the ideals which are Gotzmann but not
lex.  My hunch is  that these are rare; perhaps there is even
something like a restriction on their Hilbert functions.  This is the
beginning of an attempt to get at that, if it exists.

Let $I$ be Gotzmann, and let $J$ be its compression.  Then $J$ is also
Gotzmann.  Let $d$ be the maximal degree in which $J$ is not lex; that
is, $J_{d+1}$ is lex.  Decompose $J_{d}=T_{1}\oplus aT_{a}$.  By
induction on the number of variables we have $T_{a}\neq Q_{d-1}$ and
$T_{1}\neq 0$.  However, $\mathbf{n}_1 T_{a}=Q_{d}$.

Hence, every monomial in $Q_{d-1}\smallsetminus T_{a}$ is divisible by
$x_{n}$.  If there are $k$ such monomials, the first of these is
$(x_{n-d+1}\dots x_{n-d+k-1})(x_{n-d+k+1}\dots x_{n})$ (the first
factor being empty if $k=1$).  In particular, $k\leq d-1$ and $T_{a}$
contains everything divisible by $x_{n-d}$.  

Now the last $k$ monomials of $T_{1}$ must be divisible by $x_{n}$ as
well, since replacing them with the missing monomials from $T_{a}$
must not change the Hilbert function.  This seems like a restriction
on the Macaulay representation of $T_{1}$.  Is it?

Similarly, if $J$ is generated in degree $d-\ell$, and we write
$J=J_{1}\oplus aJ_{a}$, then all the missing monomials of $J_{a}$ are
divisible by $x_{n-\ell}$ or some later variable, and $J_{a}$ contains
everything divisible by $x_{n-d}$.  If $J_{a}$ is missing $k$
monomials, then the last $k$ monomials of $J_{1}$ are all divisible by
$x_{n-\ell}$ or some later variable; in fact, the number of these
monomials with $\max(m)=x_{p}$ is equal to the number of the missing
monomials from $J_{a}$ with $\max(m)=x_{p}$ for all $p$.  Again, this
seems to call for binomial coefficients.  I don't trust myself with them.
}
\providecommand{\bysame}{\leavevmode\hbox to3em{\hrulefill}\thinspace}
\providecommand{\MR}{\relax\ifhmode\unskip\space\fi MR }
\providecommand{\MRhref}[2]{%
  \href{http://www.ams.org/mathscinet-getitem?mr=#1}{#2}
}
\providecommand{\href}[2]{#2}

\end{document}